\documentclass[a4,12pt]{amsart}
\usepackage{amssymb,amstext,amsmath,amscd,amsthm,amsfonts,enumerate,latexsym}
\usepackage{color}
\usepackage[dvipdfmx]{graphicx}
\usepackage[all]{xy}
\theoremstyle{plain}
\newtheorem{thm}{Theorem}[section]
\newtheorem{theorem}[thm]{Theorem}
\newtheorem*{thm*}{Theorem}
\newtheorem*{cor*}{Corollary}

\newtheorem{prop}[thm]{Proposition}
\newtheorem{proposition}[thm]{Proposition}
\newtheorem{lemma}[thm]{Lemma}

\newtheorem{cor}[thm]{Corollary}
\newtheorem{corollary}[thm]{Corollary}
\newtheorem{claim}{Claim}
\newtheorem*{claim*}{Claim}

\theoremstyle{definition}
\newtheorem{defn}[thm]{Definition}

\newtheorem{ex}[thm]{Example}
\newtheorem{Example}[thm]{Example}

\newtheorem{remark}[thm]{Remark}

\newtheorem{ques}[thm]{Question}

\theoremstyle{remark}

\numberwithin{equation}{thm}

\def\Min{\operatorname{Min}}

\def\Ext{\operatorname{Ext}}

\def\Im{\operatorname{Im}}

\def\Hom{\operatorname{Hom}}

\def\mod{\mathrm{mod}}

\def\Coker{\mathrm{Coker}}

\def\rank{\mathrm{rank}}

\def\m{\mathfrak m}
\def\n{\mathfrak n}

\newcommand{\Ann}{\mathrm{Ann}}

\newcommand{\rme}{\mathrm{e}}

\newcommand{\rmr}{\mathrm{r}}

\newcommand{\rmE}{\mathrm{E}}

\newcommand{\rmK}{\mathrm{K}}

\newcommand{\rmQ}{\mathrm{Q}}

\newcommand{\calF}{\mathcal{F}}

\newcommand{\calX}{\mathcal{X}}

\newcommand{\fka}{\mathfrak{a}}

\newcommand{\fkc}{\mathfrak{c}}

\newcommand{\fkm}{\mathfrak{m}}

\newcommand{\fkp}{\mathfrak{p}}
\newcommand{\fkq}{\mathfrak{q}}

\newcommand{\mapright}[1]{%
\smash{\mathop{%
\hbox to 1cm{\rightarrowfill}}\limits^{#1}}}

\newcommand{\mapleft}[1]{%
\smash{\mathop{%
\hbox to 1cm{\leftarrowfill}}\limits_{#1}}}

\def\depth{\operatorname{depth}}

\def\Spec{\operatorname{Spec}}

\title[The Cohen-Macaulay type of idealizations]{Residually faithful modules and the Cohen-Macaulay type of idealizations}

\author{Shiro Goto}
\address{Department of Mathematics, School of Science and Technology, Meiji University, 1-1-1 Higashi-mita, Tama-ku, Kawasaki 214-8571, Japan}
\email{shirogoto@gmail.com}

\author{Shinya Kumashiro}
\address{Department of Mathematics and Informatics, Graduate School of Science and Technology, Chiba University, Chiba-shi 263, Japan}
\email{polar1412@gmail.com}

\author{Nguyen Thi Hong Loan}
\address{Department of Mathematics, School of Natural Sciences Education, Vinh University, 182 Le Duan, Vinh City, Nghe An Province, Vietnam}
\email{nhloandhv@gmail.com}


\thanks{The first author was partially supported by the JSPS Grant-in-Aid for Scientific Research (C) 16K05112. The first and the second authors were partially supported by Bilateral Programs (Joint Research) of JSPS and International Research Supporting Programs of Meiji University. The third author was partially supported by International Research Supporting Programs of Meiji University.}

\begin{document}
\maketitle

\setlength{\baselineskip} {17pt}

%{\footnotesize{\tableofcontents}}

\begin{abstract}
The Cohen-Macaulay type of idealizations of maximal Cohen-Macaulay modules over  Cohen-Macaulay local rings is explored. There are two extremal cases, one of which is closely related to the theory of Ulrich modules \cite{BHU, GOTWY1, GOTWY2, GTT2}, and the other one is closely related to the theory of residually faithful modules and the theory of closed ideals \cite{BV}.

\end{abstract}

\section{Introduction}
The purpose of this paper is to explore the behavior of the Cohen-Macaulay type of idealizations of maximal Cohen-Macaulay modules over Cohen-Macaulay local rings, mainly  in connection with their residual faithfulness.

Let $R$ be a commutative ring and $M$ an $R$-module. We set $A = R \oplus M$ as an additive group and define the multiplication in $A$ by
$$(a,x) {\cdot} (b, y) = (ab, ay + bx) $$
for $(a,x), (b, y) \in A$. Then, $A$ forms a commutative ring, which we denote by $A=R\ltimes M$ and call the idealization of $M$ over $R$ (or, the trivial extension of $R$ by $M$). Notice that $R \ltimes M$ is a Noetherian ring if and only if so is the ring $R$ and the $R$-module $M$ is  finitely generated. If $R$ is a local ring with maximal ideal $\m$, then so is the idealization $A=R \ltimes M$, and the maximal ideal $\n$ of $A$ is given by $\n = \m \times M$.

The notion of the idealization was introduced in the book \cite{N} of Nagata, and we now have diverse applications in several directions (see, e.g., \cite{AW, GMP, GTT}). Let $(R,\m)$ be a Cohen-Macaulay local ring of dimension $d$. We set $$\rmr(R) = \ell_R\left(\Ext_R^d(R/\m,R)\right)$$ and call it the Cohen-Macaulay type of $R$ (here $\ell_R(*)$   denotes the length). Then, as is well-known, $R$ is a Gorenstein ring if and only if $\rmr(R) = 1$, so that the invariant $\rmr(R)$ measures how different the ring $R$ is from being a Gorenstein ring. In the current paper, we are interested in the Cohen-Macaulay type $\rmr(R \ltimes M)$ of $R\ltimes M$, for a maximal Cohen-Macaulay (MCM for short) $R$-module $M$, that is a finitely generated $R$-module $M$ with $\depth_RM=\dim R$. In the researches of this direction, one of the most striking results is, of course, the characterization of canonical modules obtained by I. Reiten \cite{R}. She showed that $R \ltimes M$ is a Gorenstein ring if and only if $R$ is a Cohen-Macaulay local ring and $M$ is the canonical module of $R$, assuming $(R,\m)$ is a Noetherian local ring and $M$ is a non-zero finitely generated $R$-module. Motivated by this result, our study aims at explicit formulae of the Cohen-Macaulay type $\rmr(R \ltimes M)$ of idealizations for diverse MCM $R$-modules $M$.

Let us state some of our main results, explaining how this paper is organized. 
Throughout, let $(R, \m)$ be a Cohen-Macaulay local ring, and $M$ a MCM $R$-module. Then, we have in general $$\rmr_R(M) \le \rmr(R \ltimes M) \le \rmr(R) + \rmr_R(M)$$ (here $\rmr_R(M) = \ell_R\left(\Ext_R^d(R/\m,M)\right)$ denotes the Cohen-Macaulay type of $M$), which we shall confirm in Section 2 (Theorem \ref{pro1}).  As is shown in Example \ref{2.3.1} and Proposition \ref{2.3.2}, the difference $\rmr(R \ltimes M)-\rmr_R(M)$ can be arbitrary among the interval $[0, \rmr(R)]$. We explore two extremal cases; one is the case of  $\rmr(R \ltimes M) = \rmr_R(M)$, and the other one is the case of $\rmr(R \ltimes M) = \rmr(R) + \rmr_R(M)$.

The former case is exactly the case where $M$ is a residually faithful $R$-module and closely related to the preceding research \cite{BV}. To explain the  relationship more precisely, for $R$-modules $M$ and $N$, let $$t=t^M_N: \Hom_R(M,N)\otimes_R M \to N$$ denote the $R$-linear map defined by $t(f\otimes x)=f(x)$ for all $f \in \Hom_R(M,N)$ and $x \in M$. With this notation, we have the following, which we will prove in Section 3. Here, $\mu_R(*)$ denotes the number of elements in a minimal system of generators.

\begin{thm}\label{1.1}
Let $M$ be a MCM $R$-module and suppose that $R$ possesses the canonical module $\rmK_R$.  Then $$\rmr(R\ltimes M) = \rmr_R(M) + \mu_R(\operatorname{Coker} t^M_{\rmK_R}).$$ 
\end{thm}

\noindent
As a consequence, we get the following, where the equivalence between Conditions (2) and (3) is due to \cite[Proposition 5.2]{BV}.  Remember that a MCM $R$-module $M$ is said to be {\it residually faithful}, if $M/\fkq M$ is a faithful $R/\fkq$-module for some (eventually, for every) parameter ideal $\fkq$ of $R$ (cf. \cite[Definition 5.1]{BV}).

\begin{cor}[cf. {\cite[Proposition 5.2]{BV}}]\label{1.2}
Let $M$ be a MCM  $R$-module and suppose that $R$ possesses the canonical module $\rmK_R$.  Then the following conditions are equivalent.
\begin{enumerate}[{\rm (1)}]
\item[$(1)$] $\rmr(R \ltimes M) = \rmr_R(M)$.
\item[$(2)$] The homomorphism $t_{\rmK_R}^M : \Hom_R(M, \rmK_R) \otimes_RM \to \rmK_R$ is surjective.
\item[$(3)$] $M$ is  a residually faithful $R$-module.
\end{enumerate}
\end{cor}

%\noindent
In Section 3, we will also show the following, where $\operatorname{\Omega CM}(R)$ denotes the class of the (not necessarily minimal) first syzygy modules of MCM $R$-modules.

\begin{thm}\label{1.3}
Let $M \in \operatorname{\Omega CM}(R)$. Then
$$
\rmr(R \ltimes M) = \begin{cases} 
\rmr_R(M)  &\text{if $R$ is a direct summand of $M$}, \\
\rmr(R) + \rmr_R(M) &\text{otherwise.}
\end{cases}
$$
\end{thm}

In Section 4, we are concentrated in the latter case where $\rmr(R \ltimes M) = \rmr(R) + \rmr_R(M)$, which is closely related to the theory of Ulrich modules (\cite{BHU, GOTWY1, GOTWY2, GTT2}). In fact, the equality $\rmr(R \ltimes M) = \rmr(R) + \rmr_R(M)$ is equivalent to saying that $(\fkq:_R\m)M = \fkq M$ for some (and hence every) parameter ideal $\fkq$ of $R$, so that all the Ulrich modules and all the syzygy modules $\Omega_R^i(R/\m)$ ~($i \ge d$) satisfy the above equality $\rmr(R \ltimes M) = \rmr(R) + \rmr_R(M)$  (Theorems \ref{1.5.1}, \ref{1.2}), provided $R$ is not a regular local ring (here $\Omega_R^i(R/\m)$ is considered in a minimal free resolution of $R/\m$).

In Section 5, we give the bound of $\sup \rmr(R \ltimes M)$, where $M$ runs through certain MCM $R$-modules. In particular, when $d=1$, we get the following (Corollary \ref{5.2}).

\begin{thm}\label{1.4}
Suppose that $(R,\m)$ is a Cohen-Macaulay local ring of dimension one and multiplicity $e$. Let $\calF$ be the set of $\m$-primary ideals of $R$. 
Then 
$$\underset{I \in \calF}{\sup}~\rmr(R \ltimes I)=
\begin{cases}
1 & \text{if $R$ is a $\operatorname{DVR}$,}\\
\rmr(R) + e & \text{otherwise.}
\end{cases}
$$
\end{thm}

In Section 6, we focus our attention on the case where $\dim R = 1$. The main objectives are the trace ideals and closed ideals. The notion of closed ideals was introduced by \cite{BV}, where one finds a beautiful theory of closed ideals. As for the theory of trace ideals, we refer to \cite{GIK, L} for the recent progress.  In Section 6, we compute the Cohen-Macaulay type $\rmr(R \ltimes I)$ for fractional trace or closed ideals $I$ over a one-dimensional Cohen-Macaulay local ring $R$, in terms of the numbers of generators of $I$ together with the Cohen-Macaulay type $\rmr_R(I)$ of $I$ as an $R$-module.

In what follows, unless otherwise specified, $(R,\m)$ denotes a Cohen-Macaulay local ring with $d = \dim R \ge 0$. When $R$ possesses the canonical module $\rmK_R$, for each $R$-module $M$ we denote $\Hom_R(M,\rmK_R)$ by $M^\vee$. Let $\rmQ(R)$ be the total ring of fractions of $R$. For $R$-submodules $X$ and $Y$ of $\rmQ(R)$, let $$X:Y = \{a \in \rmQ(R) \mid aY \subseteq X\}.$$
If we  consider ideals $I,J$ of $R$, we set $I:_RJ=\{a \in R \mid aJ \subseteq I\}$; hence $$I:_RJ = (I:J) \cap R.$$ For each finitely generated $R$-module $M$, let $\mu_R(M)$ (resp. $\ell_R(M)$) denote the number of elements in a minimal system of generators (resp. the length) of $M$. For an $\m$-primary ideal $\fka$ of $R$, we denote by $$\rme_\fka^0(M)= \lim_{n \to \infty}d!{\cdot}\displaystyle\frac{\ell_R(M/\fka^{n}M)}{n^d}$$
 the multiplicity of $M$ with respect to $\fka$.

%%%%%%%%%%%%%%%%%%%%%%%%%%%%%%%%%%%%%%%%%%%%%%%%%%%%%%%%%%%%%%%%%%%%%%%%%%%%%%%%%%%%%%%%%%%%%%%%%%%%%%%%%%%%%%%%%%%%%%%%%%%%%%%%%%%%%%%%%%%%%%%%%%%%%%%%%%%%%%%%%%%%%%%%%%%%%%%%%%%%%%%%%%%%%%%%%%%%%%%%%%%%%%%%%%%%%%%%%%%%%%%%%%%%%%%%%%%%

\section{The Cohen-Macaulay type of general idealizations}
In this section, we estimate the Cohen-Macaulay type of idealizations for general maximal Cohen-Macaulay modules over Cohen-Macaulay local rings. We begin with the following observation, which is the starting point of this research.

\begin{prop} \label{lemma1}
Let $(R,\m)$ be a $($not necessarily Noetherian$)$ local ring and let $M$ be an $R$-module. We set $A = R \ltimes M$ and denote by $\n = \m \times M$ the maximal ideal of $A$.  Then
$$(0):_A \n= \left(\left[(0):_R\m\right] \cap \Ann_RM\right) \times \left[(0):_M\m\right].$$
 Therefore, when $R$ is an Artinian local ring, $(0):_A \n= (0) \times \left[(0):_M\m\right]$ if and only if $\Ann_RM = (0)$.
\end{prop}

\begin{proof} Let $(a,x) \in A$. Then $(a,x){\cdot}(b,y)=0$ for all $(b,y) \in \n=\m \times M$ if and only if $ab=0, ay =0$, and $bx=0$ for all $b \in \m, y \in M$. Hence, the first equality follows. Suppose that $R$ is an Artinian local ring. Then, since $I= \Ann_RM$ is an ideal of $R$, $I  \ne (0)$ if and only if $[(0):_R\m] \cap I \ne (0)$, whence the second assertion follows.
\end{proof}

We now assume, throughout this section, that $(R,\m)$ is a Cohen-Macaulay local ring with $d = \dim R \ge 0$. We say that a finitely generated $R$-module $M$ is a {\it maximal Cohen-Macaulay} (MCM for short) $R$-module, if $\depth_RM=d$.

\begin{thm} \label{pro1}  Let $M$ be a MCM $R$-module and  $A = R \ltimes M$. Then 
$$\rmr_R(M) \le \rmr(A)  \le \rmr(R) + \rmr_R(M).$$ Let $\fkq$ be a parameter ideal of $R$  and set $\overline{R} = R/\fkq$, $\overline{M} = M/\fkq M$.  We then have the following. 
\begin{enumerate}
\item[$(1)$] $\rmr(A) =  \rmr_R(M)$ if and only if $\overline{M}$ is a faithful $\overline{R}$-module.
\item[$(2)$] $\rmr(A) =  \rmr(R) + \rmr_R(M)$ if and only if $(\fkq :_R\m)M = \fkq M$.
\end{enumerate}
\end{thm}

\begin{proof} We set $\overline A = A/\fkq A.$ Therefore, $\overline A = \overline R \ltimes \overline M$.  
Since $A$ is a Cohen-Macaulay local ring and $\fkq A$ is a parameter ideal of $A$, we have  $\rmr(A) = \rmr(\overline A)$, and by Proposition \ref{lemma1} it follows that 
$$\begin{aligned}\rmr(A) = \ell_{\overline A} ((0):_{\overline A}\n) & =\ell_{\overline A} (\left(\left[(0):_{\overline R}\m\right] \cap \Ann_{\overline{R}}\overline{M}\right) \times \left[(0):_{\overline M}\m\right])\\
& = \ell_{\overline R}(\left[(0):_{\overline{R}}\m \right]\cap \Ann_{\overline{R}}\overline{M})  + \ell_{\overline R}\left((0):_{\overline M}\m\right) \\
& = \ell_{\overline R}(\left[(0):_{\overline{R}}\m\right]\cap \Ann_{\overline{R}}\overline{M})  + \rmr_R(M) \\
& \leqslant \ell_{\overline R}((0):_{\overline{R}}\m)  + \rmr_R(M) \\
& = \rmr(R)  + \rmr_R(M).
\end{aligned}$$
Hence, $\rmr_R(M) \le \rmr(A) \le \rmr(R) + \rmr_R(M)$, so that  by Proposition \ref{lemma1},  $\rmr(A) =  \rmr_R(M)$ if and only if $\overline{M}$ is a faithful $\overline{R}$-module. We have $\rmr(A)  =  \rmr(R) + \rmr_R(M)$ if and only if $(0):_{\overline{R}}\overline\m \subseteq \Ann_{\overline{R}}\overline{M}$, and the latter condition is equivalent to saying that $\fkq :_R \m \subseteq \fkq M: _R M$, that is $(\fkq :_R\m)M = \fkq M$.
\end{proof}

The following shows the difference $\rmr(A) - \rmr_R(M)$ in Theorem \ref{pro1} can be arbitrary among the interval $[0, \rmr(R)]$. Notice that $\rmr(R \ltimes R) = \rmr(R)$.

\begin{ex}\label{2.3.1}
Let $\ell \ge 2$ be an integer and $S= k[[X_1, X_2, \ldots, X_\ell]]$ the formal power series ring over a field $k$. Let $\fka = \Bbb I_2(\Bbb M)$ denote the ideal of $S$ generated by the maximal minors of the matrix $\Bbb M = \left(\begin{smallmatrix}
X_1&X_2& \ldots& X_{\ell-1}&X_\ell\\
X_2& X_3& \ldots&X_\ell&X_1^q\\
\end{smallmatrix}\right)
$
with $q \ge 2$. We set $R=S/\fka$. Then $R$ is a Cohen-Macaulay local ring of dimension one. For each integer $2 \le p \le \ell$, we consider the ideal $I_p= (x_1)+ (x_p, x_{p+1}, \ldots, x_\ell)$ of $R$, where $x_i$ denotes the image of $X_i$ in $R$. Then $\rmr(R \ltimes I_p)= (\ell-p+1) + \rmr_R(I_p)$, and $$\rmr_R(I_p)= \begin{cases} 
\ell  &\text{if $p=2$} \\
\ell-1 &\text{if $p \ge 3$}
\end{cases}
$$
for each $2 \le p \le \ell$.  
\end{ex}

\begin{proof}
Let $\m$ denote the maximal ideal of $R$.
We set $I = I_p$ and $x = x_1$. It is direct to check that $I^2 = xI$, where we  use the fact that $q \ge 2$. In particular, $\m^2 = x \m$.  We consider the exact sequence $$(E)\ \ \ 0 \to R/I \overset{\iota}{\to} I/xI \to I/(x) \to 0,$$
where $\iota(1) = x ~\mod ~xI$, and get $\Ann_RI/xI = I$, since $I^2 = xI$. Therefore, $\Ann_{R/(x)}I/xI = I/(x)$. Because $I/(x) \subseteq \m/(x) = (0):_{R/(x)}\m$, we get $$\ell_R([(0):_{R/(x)}\m] \cap \Ann_{R/(x)}I/xI) = \ell_R(I/(x)) = \ell-p+1,$$ whence
$$\rmr(R \ltimes I) =(\ell-p+1)+ \rmr_R(I)$$ by Theorem \ref{pro1}. Because $(x_2, x_3, \ldots, x_{p-1}){\cdot}(x_p, x_{p+1}, \ldots, x_\ell) \subseteq xI$, the above sequence $(E)$ remains exact on the socles, so that $$\rmr_R(I) = \rmr(R/I) + \rmr_R(I/(x)).$$ Therefore, $\rmr_R(I)=\ell$ if $p=2$, and $\rmr_R(I)= (p-2)+(\ell-p+1) = \ell-1$ if $p \ge 3$.
\end{proof}

Assume that $R$ is not a regular local ring and let $0 \le n \le \rmr(R)$ be an integer. Then, we suspect if there exists a MCM $R$-module $M$ such that $\rmr(R \ltimes M) = n + \rmr_R(M)$. When $R$ is the semigroup ring of a numerical semigroup, we however have an affirmative  answer. 

\begin{prop}\label{2.3.2}
Let $a_1, a_2, \ldots, a_\ell$ be positive integers such that $\operatorname{GCD}(a_1, a_2,\cdots,a_\ell)=1$. Let $H =\left<a_1, a_2, \ldots, a_\ell \right>$ be the numerical semigroup generated by $\{a_i\}_{1 \le i \le \ell}$. Let $k[[t]]$ denote the formal power series ring over a field $k$ and consider, inside of $k[[t]]$,  the semigroup ring $$R= k[[t^{a_1}, t^{a_2}, \ldots, t^{a_\ell}]]$$ of $H$ over $k$. We set $e = \min \{a_i \mid 1 \le i \le \ell\}$ and assume that $e > 1$, that is $R$ is not a $\operatorname{DVR}$. Let $r = \rmr(R)$. Then, for each integer $0 \le n \le r$, $R$ contains a non-zero ideal $I$ such that $\rmr(R \ltimes I) = n + \rmr_R(I)$.
\end{prop}

\begin{proof}

Let $\m$ be the maximal ideal of $R$ and set $B =\m : \m$. Then $B = R : \m$ since $R$ is not a DVR, and  $$(t^e):_R\m = (t^e):\m = t^e(R:\m)= t^eB.$$ We denote by $\operatorname{PF}(H) = \{\alpha_1 < \alpha_2 < \cdots < \alpha_r\}$ the pseudo-Frobenius numbers of $H$. Hence, $B = R + \sum_{1\le i \le r}Rt^{\alpha_i}$, so that $(t^e):_R\m= (t^e)+(t^{\alpha_i + e} \mid 1 \le i\le r)$. Let $1 \le p \le r$ be an integer and set $I = (t^e)+(t^{\alpha_j + e} \mid p \le j \le r) \subseteq (t^e):_R\m$. Let $\alpha_0 = 0$. We then have the following.

\begin{claim}\label{claim 1}
Let $0 \le i \le r$ and $p \le j \le r$ be integers. Then $t^{\alpha_i + e}t^{\alpha_j +e} \in t^eI$. Consequently, $I^2 = t^eI$. 
\end{claim}

\begin{proof}
Assume that  $t^{\alpha_i + e}t^{\alpha_j +e} \not\in t^eI$. Then $t^{\alpha_i + \alpha_j + e} \not\in I$. On the other hand, since $t^{\alpha_i}t^{\alpha_j} \in B=\m:\m$, we get $\alpha_i + \alpha_j = \alpha_k + h$ for some $0 \le k \le r$ and $h \in H$. If $h >0$, then $\alpha_i +\alpha_j \in H$, so that $t^{\alpha_i + \alpha_j + e} \in I$, which is impossible. Therefore, $h = 0$, and $\alpha_k - \alpha_j = \alpha_i \ge 0$, so that $k \ge j \ge p$. Hence, $t^{\alpha_i + \alpha_j + e} = t^{\alpha_k + e} \in I$. This is a contradiction.
\end{proof}

We now consider the exact sequence $0 \to R/I \to I/t^eI \to I/(t^e) \to 0$, and get that $\operatorname{Ann}_RI/t^eI= I$. Hence $$\operatorname{Ann}_{R/(t^e)}I/t^eI = I/(t^e) \subseteq (0):_{R/(t^e)}\m.$$  Therefore, $\rmr(R \ltimes I) = \ell_R(I/(t^e)) + \rmr_R(I) = n+\rmr_R(I)$, where $n=r - p+1$. For $n=0$, just take $I=R$. 
\end{proof}

\begin{remark} With the same notation as in the proof of Proposition \ref{2.3.2}, let $\rmK_R$ denote the canonical module of $R$ and consider the ideal $I = (t^e)+(t^{\alpha_j + e} \mid p \le j \le r)$. Then, because $I^2 = t^eI$ and $\m I=\m t^e$, by \cite[Proposition 6.1]{GMP} $R \ltimes I^\vee$ is an almost Gorenstein local ring, where $I^\vee = \Hom_R(I, \rmK_R)$. Since $\operatorname{Ann}_RI^\vee/t^eI^\vee= \operatorname{Ann}_R I/t^eI$, we get $$\rmr(R\ltimes I^\vee) = (r-p+1) + \rmr_R(I^\vee)= (r-p+1) + \mu_R(I),$$ so that $\rmr(R \ltimes I^\vee) = 2r-2p +3$.
\end{remark}

\begin{cor}\label{2.3.3}
With the same notation as in Proposition $\ref{2.3.2}$, assume that $a_1<a_2<\cdots <a_\ell$, and that $H$ is minimally generated by $\ell$ elements with $\ell = a_1 \ge 2$,  that is $R$ has maximal embedding dimension $\ell \ge 2$. Let $2 \le p \le \ell$ be an integer and set $I_p = (t^{a_1})+(t^{a_p},t^{a_{p+1}},\ldots, t^{a_\ell})$. Then $\rmr(R \ltimes I_p)= (\ell-p+1) + \rmr_R(I_p)$, and $$\rmr_R(I_p)= \begin{cases} 
\ell  &\text{if $p=2$} \\
\ell-1 &\text{if $p \ge 3$}
\end{cases}
$$
for each $2 \le p \le \ell$.  
\end{cor}

\begin{proof}
Let $e = a_1$ and $r = \rmr(R)$. Hence $\rmr(R) = e -1$. Let $1 \le i,j \le \ell$ be integers. Then $i=j$ if $a_i \equiv a_j$ ~$\mod$ $e$, because $H$ is minimally generated by $\{a_i\}_{1 \le i \le \ell}$. Therefore,  $\operatorname{PF}(H) = \{a_2-e < a_3-e < \cdots < a_{e}-e\}$, so that $\rmr(R\ltimes I_p) = ( e-p+1) + \rmr_R(I_p)$ by Proposition \ref{2.3.2}. To get $\rmr_R(I_p)$, by the proof of Example \ref{2.3.1} it suffices to show that $\m {\cdot}(t^{a_p}, t^{a_{p+1}}, \ldots, t^{a_\ell}) \subseteq t^{a_1}I$, which follows from Claim \ref{claim 1} in the proof of Proposition \ref{2.3.2}.
\end{proof}

In the following two sections, Sections 3 and 4, we explore the extremal cases where $\rmr(R\ltimes M) = \rmr_R(M)$ and $\rmr(R \ltimes M) = \rmr(R)+\rm_R(M)$, respectively.

\section{Residually faithful modules and the case where $\rmr(R \ltimes M) = \rmr_R(M)$}
Let $(R,\m)$ be a Cohen-Macaulay local ring with $d = \dim R \ge 0$. In this section, we consider the case of Theorem \ref{pro1} (1), that is $\rmr(R \ltimes M) = \rmr_R(M)$. Let us  begin with the following.

\begin{defn}\label{1.5.1} Let $M$ be a MCM $R$-module. We say that $M$ is {\em residually faithful}, if $M/\fkq M$ is a faithful $R/\fkq$-module for some parameter ideal $\fkq$ of $R$.
\end{defn}

\noindent 
With this definition, Theorem \ref{pro1} (1) assures the following.

\begin{prop}\label{1.5.2}
Let $M$ be a MCM $R$-module. Then the following conditions are equivalent.
\begin{enumerate}[{\rm (1)}]
\item $\rmr(R \ltimes M) =\rmr_R(M)$.
\item $M$ is  a residually faithful $R$-module.
\item $M/\fkq M$ is a faithful $R/\fkq$-module for every parameter ideal $\fkq$ of $R$.
\end{enumerate}
\end{prop}

For $R$-modules $M$ and $N$, let $$t= t_{N}^M: \Hom_R(M,N)\otimes_R M \to N$$ denote the $R$-linear map defined by $t(f\otimes m)=f(m)$ for all $f \in \Hom_R(M,N)$ and $m \in M$. With this notation, we have the following.

\begin{thm}\label{1.5.3}
Let $M$ be a MCM $R$-module and suppose that $R$ possesses the canonical module $\rmK_R$.  Let $C = \Coker~t_{\rmK_R}^M$. Then $$\rmr(R\ltimes M) = \rmr_R(M) + \mu_R(C).$$ 
\end{thm}

\begin{proof}
We set $K = \rmK_R$ and $A = R\ltimes M$. Let us make the $R$-module $M^\vee \times K$ into an $A$-module on which the $A$-action is defined by $$(a,m)\circ (f,x) = (af, f(m) +ax)$$
for each $(a,m) \in A$ and $(f,x) \in M^\vee \times K$. Then $M^\vee \times K \cong \Hom_R(A,K)$ as an $A$-module. Therefore, $\rmK_A = M^\vee \times K$, the canonical module of $A$ (\cite[Section 6, Augmented rings]{GGHV} or \cite[Section 2]{GK}). Let $\n = \m \times M$ denote the maximal ideal of $A$ and $L = \operatorname{Im} t_{\rmK_R}^M$. Then, since $\n (M^\vee \times \rmK_R) = \m M^\vee \times (L + \m \rmK_R)$, we get
  
$$
\begin{aligned}
\rmr(A)&=\mu_A(\rmK_A)\\
&=\ell_A([M^\vee \times K]/[\m M^\vee \times (L + \m K)]\\
&=\ell_R([M^\vee \oplus K]/[\m M^\vee \oplus (L + \m K)] \\
&= \ell_R(M^\vee/\m M^\vee) + \ell_R(K/(L + \m K))\\
&= \mu_R(M^\vee) + \mu_R(C)\\
&= \rmr_R(M) + \mu_R(C).\\
\end{aligned}
$$
\end{proof}

Theorem \ref{1.5.3} covers \cite[Proposition 5.2]{BV}. In fact, we have the following, where the equivalence of Conditions (1) and (3) follows from Proposition \ref{1.5.2}, and the equivalence of Conditions (1) and (2) follows from Theorem \ref{1.5.3}.

\begin{cor}[cf. {\cite[Proposition 5.2]{BV}}]\label{1.4}
Let $M$ be a MCM $R$-module and suppose that $R$ possesses the canonical module $\rmK_R$.  Then the following conditions are equivalent.
\begin{enumerate}[{\rm (1)}]
\item[$(1)$] $\rmr(R \ltimes M) = \rmr_R(M)$.
\item[$(2)$] The homomorphism $t_{\rmK_R}^M : \Hom_R(M, \rmK_R) \otimes_RM \to \rmK_R$ is surjective.
\item[$(3)$] $M$ is  a residually faithful $R$-module.
\end{enumerate}
\end{cor}

We note one example of residually faithful modules $M$ such that $M \not\cong R, \rmK_R$. 

\begin{ex}[{\cite[Example 7.3]{GTT3}}]
Let $k[[t]]$ be the formal power series ring over a field $k$ and consider $R= k[[t^9, t^{10}, t^{11}, t^{12}, t^{15}]]$ in $k[[t]]$. Then $\rmK_R=R + Rt+Rt^3+Rt^4$ and $\mu_R(\rmK_R)=4$.  Let $I = R + Rt$. Then the homomorphism $t^I_{\rmK_R}: \Hom_R(I,\rmK_R) \otimes_R I \to \rmK_R$ is an isomorphism of $R$-modules, so that $I$ is a residually faithful $R$-module, but $I \not\cong R, \rmK_R$, since $\mu_R(I) = 2$.
\end{ex}

Here we notice that Corollary \ref{1.4} recovers  the theorem of Reiten \cite{R} on Gorenstein modules. In fact, with the same notation as in Corollary \ref{1.4}, suppose that $R \ltimes M$ is a Gorenstein ring and let $\fkq$ be a parameter ideal of $R$. Then, since $\rmr(R\ltimes M)=1$, Corollary  \ref{1.4} implies that $\overline{M}=M/\fkq M$ is a faithful module over the Artinian local ring $\overline{R}=R/\fkq$ with  $\rmr_{\overline{R}}(\overline{M})=1$. Therefore, $\overline{M}$ is the injective envelope $\rmE_{\overline{R}}(R/\m)$ of the residue class field $R/\m$ of $\overline{R}$, so that $M \cong \rmK_R$ is the canonical module (that is a Gorenstein module of rank one) of $R$.

Residually faithful modules enjoy good properties. Let us summarize some of them.

\begin{prop}\label{1.5.4} Let $M$ be a MCM $R$-module. Then the following assertions hold true.
\begin{enumerate}[{\rm (1)}]
\item Let $a \in \m$ be a non-zerodivisor of $R$. Then $M$ is a residually faithful $R$-module if and only if so is the $R/(a)$-module $M/aM$. 
\item Let $(S,\n)$ be a Cohen-Macaulay local ring and let $\varphi : R \to S$ denote a flat local homomorphism of local rings. Then $M$ is a residually faithful $R$-module if and only if so is the $S$-module $S \otimes_RM$. Therefore, $M$ is a residually faithful $R$-module if and only if so is the $\widehat{R}$-module $\widehat{M}$, where $\widehat{*}$ denotes the $\m$-adic completion.
\item Suppose that $M$ is a residually faithful $R$-module. Then $M$ is a faithful $R$-module and $M_\fkp$ is a residually faithful $R_\fkp$-module for every $\fkp \in \Spec R$.
\end{enumerate}
\end{prop}

\begin{proof}
(1) This directly follows from Proposition \ref{1.5.2}.

(2) We set $n = \dim S/\m S$ and $L = S \otimes_RM$.  Firstly, suppose that $n = 0$. Let $\fkq$ be a parameter ideal of $R$ and set $\fka = \Ann_RM/\fkq M$. Then $\fka S = \Ann_S(L/\fkq L)$. If $\fka = \fkq$, then $\fkq S = \Ann_SL/\fkq L$, so that $L$ is a residually faithful $S$-module, since $\fkq S$ is a parameter ideal of $S$. Conversely, suppose that $L$ is  a residually faithful $S$-module. We then have $\fka S = \fkq S$ by Proposition \ref{1.5.2}, so that $\fka = \fkq$, and $M$ is a residually faithful $R$-module.

We now assume that $n >0$ and that Assertion (2) holds true for $n-1$. Let $g \in \n$  and suppose that $g$ is $S/\m S$-regular. Then $g$ is $S$-regular and the composite homomorphism $$R \to S \to S/gS$$ remains flat and local, so that $M$ is a residually faithful $R$-module if and only if so is the $S/gS$-module $L/gL$. Since $\dim S/(gS + \m S) = n-1$, the latter condition is, by Assertion (1), equivalent to saying that $L$ is a residually faithful $S$-module.

(3) Let $a_1, a_2, \ldots, a_d$ be a system of parameters of $R$. We then have by Proposition \ref{1.5.2} 
$$\Ann_RM \subseteq \Ann_RM/(a_1^n, a_2^n, \ldots, a_d^n)M =(a_1^n, a_2^n, \ldots, a_d^n)$$ for all $n >0$. Therefore, $M$ is a faithful $R$-module. Let $\fkp \in \Spec R$ and choose $P \in \Min_{\widehat{R}}\widehat{R}/\fkp \widehat{R}$. Then, $\fkp = P \cap R$, and we get a flat local homomorphism $R_\fkp \to \widehat{R}_P$ of local rings such that $\dim \widehat{R}_P/\fkp \widehat{R}_P= 0$. Therefore, to see that $M_\fkp$ is a residually faithful $R_\fkp$-module, by Assertion (1) it suffices to show that $\widehat{M}_P$ is a residually faithful $\widehat{R}_P$-module. Consequently,  because $\widehat{M}$ is a residually faithful $\widehat{R}$-module by Assertion (1), passing to the $\m$-adic completion $\widehat{R}$ of $R$, without loss of generality we may assume that $R$ possesses the canonical module $\rmK_R$. Then, the current assertion readily follows from Corollary \ref{1.4}, because $$\rmK_{R_\fkp}= (\rmK_R)_\fkp= \left(\operatorname{Im} t_{\rmK_R}^M\right)_\fkp = \operatorname{Im} t_{\rmK_{R_\fkp}}^{M_\fkp}.$$

\end{proof}

By Proposition \ref{1.5.4},  we have the following.

\begin{cor}
Let $M$ be a MCM $R$-module. If $\rmr(R\ltimes M) = \rmr_R(M)$, then $\rmr(R_\fkp \ltimes M_\fkp) = \rmr_{R_\fkp}(M_\fkp)$ for every $\fkp \in \Spec R$.
\end{cor}

\begin{cor} Let $M$ be a MCM $R$-module, and suppose that $R$ possesses the canonical module $\rmK_R$. If $M$ is a residually faithful $R$-module, then so is $M^\vee$.
\end{cor}

\begin{proof}
We may assume that $d >0$ and that our assertion holds true for $d-1$. Let $a\in \m$ be a non-zerodivisor of $R$ and let $\overline{*}$ denote the reduction mod $(a)$. We then have $\overline{M^\vee} \cong \Hom_{\overline{R}}(\overline{M}, \overline{\rmK_R})=\overline{M}^\vee$, where we identify $\overline{\rmK_R} = \rmK_{\overline{R}}$. Because by Proposition \ref{1.5.4} (3), $\overline{M}$ is a residually faithful $\overline{R}$-module, by the hypothesis of induction we have  $\overline{M}^\vee = \Hom_{\overline{R}}(\overline{M}, \rmK_{\overline{R}})$ is a residually faithful $\overline{R}$-module, whence Proposition \ref{1.5.4} (1) shows that $M^\vee$ is a residually faithful $R$-module.  
\end{proof}

Suppose that  $R$ possesses the canonical module $\rmK_R$. Then, certain residually faithful $R$-modules $M$ satisfy the condition $\Hom_R(M,\rmK_R)\otimes_RM \cong \rmK_R$, as we show in the following. Recall that a finitely generated $R$-module $C$ is called {\it semidualizing}, if the natural homomorphism $R \to \Hom_R(C,C)$ is an isomorphism and $\Ext^i_R(C,C) = (0)$ for all $i > 0$. Hence,  the canonical module is semidualizing, and all the semidualizing $R$-modules satisfy the hypothesis in Theorem \ref{3.9}, because semidualizing modules are Cohen-Macaulay.

\begin{thm}\label{3.9}
Suppose that $R$ possesses the canonical module $\rmK_R$ and let $M$ be a MCM $R$-module. If $R\cong \Hom_R(M,M)$ and $\Ext_R^i(M,M) = (0)$ for all $1 \le i \le d$, then the homomorphism $$M^\vee \otimes_RM \overset{t}{\to} \rmK_R$$
is an isomorphism of $R$-modules, where $t = t^M_{\rmK_R}$. 
\end{thm}

\begin{proof}
Notice that $M$ is a residually faithful $R$-module. In fact, the assertion is clear, if $d = 0$. Suppose that $d > 0$ and let $f \in \m$ be a non-zerodivisor of $R$. We set $\overline{R} = R/(f)$ and denote $\overline{*}=\overline{R} \otimes_R*$. Then, since $f$ is regular also for $M$, we have $\Ext_R^i(M,\overline{M}) = \Ext_{\overline{R}}^i(\overline{M}, \overline{M})$ for all $i \in \Bbb Z$, and it is standard to show that $\overline{R} \cong \Hom_{\overline{R}}(\overline{M},\overline{M})$ and that $\Ext_{\overline{R}}^i(\overline{M},\overline{M}) = (0)$ for all $1 \le i \le d-1$.  Therefore, by induction on $d$, we may assume that $\overline{M}$ is a residually faithful $\overline{R}$-module, whence  Proposition \ref{1.5.4} (1) implies that so is the $R$-module $M$.

We now consider the exact sequence
$$(E) \ \ \ 0 \to X \to M^\vee \otimes_RM \overset{t}{\to} \rmK_R \to 0$$
of $R$-modules, where $t = t^M_{\rmK_R}$. If $d = 0$, then because $$\Hom_R(M^\vee \otimes_RM, \rmK_R) = \Hom_R(M, M^{\vee \vee}) = \Hom_R(M,M),$$
taking the $\rmK_R$-dual of $(E)$, we get the exact sequence $$0 \to R \to \Hom_R(M,M) \to X^\vee \to 0.$$ Hence $X^\vee = (0)$ because $R \cong \Hom_R(M,M)$, so that $M^\vee \otimes_RM \overset{t}{\to} \rmK_R$ is an isomorphism. Suppose that $d >0$ and let $f \in \m$ be $R$-regular. We denote $\overline{*} = R/(f)\otimes_R*$. Then since $f$ is $\rmK_R$-regular, we get from Exact sequence $(E)$ 
$$(\overline{E})\ \ \ 0 \to \overline{X} \to \overline{M^\vee \otimes_RM} \overset{\overline{t}}{\to} \overline{\rmK_R} \to 0.$$
Because $\overline{\rmK_R} = \rmK_{\overline{R}}$, $\overline{M^\vee \otimes_RM} = \overline{M}^\vee \otimes_{\overline{R}}\overline {M}$, and $\overline{t} = t^{\rmK_{\overline{R}}}_{\overline{M}}$, by induction on $d$ we see in the above exact sequence $(\overline{E})$ that $\overline{X}= (0)$, whence $X = (0)$ by Nakayama's lemma. Therefore, $M^\vee \otimes_RM \overset{t}{\to} \rmK_R$ is an isomorphism.  
\end{proof}

Therefore, we have the following, which guarantees that the converse of Theorem \ref{3.9} also holds true, if $R_\fkp$ is a Gorenstein ring for every $\fkp \in \Spec R \setminus \{\m\}$. See  \cite[Proposition 2.4]{GT} for details.

\begin{cor}[{\cite[Proposition 2.2]{GT}}] With the same hypothesis of Theorem $\ref{3.9}$, one has $\rmr(R)=\rmr_R(M){\cdot}\mu_R(M)$. Consequently, the following assertions hold true.
\begin{enumerate}[{\rm (1)}]
\item If $\rmr(R)$ is a prime number, then $M\cong R$ or $M \cong \rmK_R$.
\item If $R$  is a Gorenstein ring, then $M \cong R$.
\end{enumerate}
\end{cor}

Let us note the following.

\begin{prop}
Suppose that $R$ is an integral domain, possessing the canonical module $\rmK_R$. Let $M$ be a MCM $R$-module and assume that $\rmr(R \ltimes M) = 2$. If $\Ext_R^i(M,M)=(0)$ for all $1 \le i \le d$, then $$M \cong \rmK_R^{\oplus 2}\ \ \text{or}\ \ M^\vee \otimes_RM \cong \rmK_R.$$ Therefore, if $\rmr(R)$ is a prime number and $M$ is indecomposable, then $\rmr(R)=2$ and $M \cong R$.
\end{prop}

\begin{proof}
Let $C = \operatorname{Coker}t^M_{\rmK_R}$. Then,  $\rmr_R(M)= \mu_R(C)=1$, or $\rmr_R(M)= 2$ and $C=(0)$, since $\rmr(R \ltimes M) = \rmr_R(M) +\mu_R(C)$ by Theorem \ref{1.5.3}. If $\rmr_R(M)=1$, then $M^\vee \cong R$, since the cyclic module $M^\vee$ is of dimension $d$ and $R$ is an integral domain. Therefore, $M \cong \rmK_R$, so that $\rmr(R \ltimes M) = 1$, which is impossible. Hence, $\rmr_R(M) = 2$, and $M$ is, by Proposition \ref{1.5.2}, a residually faithful $R$-module. Let us take a presentation
$$0 \to X \to R^{\oplus 2} \to M^\vee \to 0$$
of $M^\vee$. If $X = (0)$, then $M \cong \rmK_R^{\oplus 2}$. Suppose that $X \ne (0)$. Then, $X$ is a MCM $R$-module, and taking the $\rmK_R$-dual of the presentation, we get the exact sequence
$$0 \to M \to \rmK_R^{\oplus 2} \to X^\vee \to 0.$$
Let $F = \rmQ(R)$. Then $F \otimes_RX^\vee \ne (0)$, since $X^\vee$ is a MCM $R$-module. Consequently, $F \otimes_RM \cong F$, that is $\rank_RM=1$, because $F \otimes_R\rmK_R \cong F$. Hence, in the canonical exact sequence
$$(E)\ \ \ 0 \to L \to M^\vee \otimes_RM \overset{t}{\to} \rmK_R \to 0,$$
 $F \otimes_RL=(0)$, because $\rank_RM=1$. Consequently, because the $R$-module $L$ is torsion, taking the $\rmK_R$-dual of the sequence $(E)$ we get the isomorphism
$$R = \rmK_R^\vee \to [M^\vee \otimes M]^\vee = \Hom_R(M,M).$$ Thus, $M^\vee \otimes_RM \cong \rmK_R$ by Theorem \ref{3.9}.

If $M$ is indecomposable and $\rmr(R)$ is a prime number, we then have $M \cong R$ or $M \cong \rmK_R$, while $\rmr(R\ltimes M) = 2$, so that $M \cong R$ and $\rmr(R) = 2$.
\end{proof}

%%%%%%%%%%%%%%%%%%
The following result is essentially due to \cite[Lemma 3.1]{T} (see also \cite[Proof of Lemma 2.2]{K}). We include a brief proof for the sake of completeness.
%%%%%%%%%%%%%%%%%%

\begin{lemma}\label{1.4.1}
Let $M$ be a MCM $R$-module and assume that there is an embedding 
$$(E)\ \ \ 0 \to M \to F \to N \to 0$$
of $M$ into a finitely generated free $R$-module $F$ such that $N$ is a MCM $R$-module.  Then the following conditions are equivalent.
\begin{enumerate}[{\rm (1)}]
\item $M$ is a residually faithful $R$-module.
\item $M \not\subseteq \m F$.
\item $R$ is a direct summand of $M$.
\end{enumerate}
\end{lemma}

\begin{proof}
(3) $\Rightarrow$ (1) and (2) $\Rightarrow$ (3) These are clear.

(1) $\Rightarrow$ (2) Let $\fkq$ be a parameter ideal of $R$. Then, since $N$ is a MCM $R$-module, Embedding (E) gives rise to the exact sequence
$$0 \to M/\fkq M \to F/\fkq F \to N/\fkq N \to 0.$$
Notice that $\Ann_{R/\fkq}\m{\cdot}(F/\fkq F) \ne (0)$ because $\dim R/\fkq = 0$, and we have $M/\fkq M \not \subseteq \m{\cdot}(F/\fkq F)$. Thus $M \not\subseteq \m F$.
\end{proof}

Let $\operatorname{\Omega CM}(R)$ denote the class of MCM $R$-modules $M$ such that there is an embedding $0 \to M \to F \to N \to 0$ of $M$ into a finitely generated free $R$-module with $N$ a MCM $R$-module. With this notation, we have the following.

\begin{thm}\label{1.4.2}
Let $M \in \operatorname{\Omega CM}(R)$. Then
$$
\rmr(R \ltimes M) = \begin{cases} 
\rmr_R(M)  &\text{if $R$ is a direct summand of $M$}, \\
\rmr(R) + \rmr_R(M) &\text{otherwise.}
\end{cases}
$$
\end{thm}

\begin{proof}
We may assume that $R$ is not a direct summand of $M$. Let us choose an embedding 
$$0 \to M \to F \to N \to 0$$
of $M$ into a finitely generated free $R$-module $F$ such that $N$ is a MCM $R$-module. Let $\fkq$ be a parameter ideal of $R$ and set $I = \fkq:_R\m$. Then, since $M \subseteq \m F$ by Lemma \ref{1.4.1}, we have from the exact sequence 
$$0 \to M/\fkq M \to F/\fkq F \to N/\fkq N \to 0$$
that $I{\cdot}(M/\fkq M )\subseteq (I\m){\cdot}(F/\fkq F) = (0)$. Therefore,  $IM \subseteq \fkq M$, so that $\rmr(R \ltimes M) = \rmr(R) + \rmr_R(M)$ by Theorem \ref{pro1} (2).
\end{proof}

If $R$ is a Gorenstein ring, every MCM $R$-module $M$ belongs to $\operatorname{\Omega CM}(R)$, so that Theorem \ref{1.4.2} yields the following.

\begin{cor}\label{gor}
Let $R$ be a Gorenstein ring and $M$ a MCM $R$-module.  Then the following conditions are equivalent.
\begin{enumerate}[{\rm (1)}]
\item $\rmr(R \ltimes M) =\rmr_R(M)$.
\item $R$ is a direct summand of $M$.
\end{enumerate}
\end{cor}

\section{Ulrich modules and the case where $\rmr(R \ltimes M) = \rmr(R) + \rmr_R(M)$}
Let $(R,\m)$ be a Cohen-Macaulay local ring of dimension $d \ge 0$. In this section, we study the other extremal case of Theorem \ref{pro1} (2), that is $\rmr(R\ltimes M)=\rmr(R)+\rmr_R(M)$. We already have a partial answer by Theorem \ref{1.4.2}, and the following also shows that over a non-regular Cohen-Macaulay local ring $(R,\m, k)$, there are plenty of MCM $R$-modules $M$ such that $\rmr(R \ltimes M) = \rmr(R) + \rmr_R(M)$.

Let $\Omega_R^i(k)$ denote, for each $i \ge 0$, the $i$-th syzygy module of the simple $R$-module $k=R/\m$ in its minimal free resolution. Notice that, thanks to Theorem \ref{1.4.2}, the crucial case in Theorem \ref{1.5.1} is actually the case where $i = d$.

\begin{thm}\label{1.5.1}
Suppose that $R$ is not a regular local ring. Then
$(\fkq:_R\m){\cdot}\Omega^i_R(k) = \fkq{\cdot} \Omega_R^i(k)$ for every $i \ge d$ and for every parameter ideal $\fkq$ of $R$. Therefore
$$\rmr(R \ltimes \Omega_R^i(k)) = \rmr(R) + \rmr_R(\Omega_R^i(k))$$ for all  $i \ge d$.
\end{thm}

\begin{proof}
We may assume that $d >0$ and that the assertion holds true for $d-1$. Choose $a \in \m \setminus \m^2$ so that $a$ is a non-zerodivisor of $R$. We set $\overline{R}= R/(a)$  and $\overline{\m}= \m/(a)$. We then have, for each $i >0$, the isomorphism  
$$ \Omega^i_R(k)/a{\cdot}\Omega^i_R(k) \cong \Omega_{\overline{R}}^{i-1}(k) \oplus \Omega_{\overline{R}}^i(k).$$ 
We now choose elements $a_2, a_3, \ldots, a_d$ of $\m$ so that $\fkq_0 = (a, a_2, a_3, \ldots, a_d)$ is a parameter ideal of $R$ and set $\overline{\fkq_0}  = \fkq_0/(a)$. Then, by the hypothesis of induction, we have
$$(\overline{\fkq_0} :_{\overline{R}} \overline{\m}){\cdot}\Omega_{\overline{R}}^i(k) = \overline{\fkq_0}{\cdot}\Omega_{\overline{R}}^i(k)$$ for all $i \ge d-1$, so that $$(\overline{\fkq_0}  :_{\overline{R}} \overline{\m}){\cdot} \left[\Omega^i_R(k)/a{\cdot}\Omega^i_R(k)\right] = \overline{\fkq_0}{\cdot} \left[\Omega^i_R(k)/a{\cdot}\Omega^i_R(k)\right]$$ for all $i \ge d$. Hence, because $\overline{\fkq_0}:_{\overline{R}}\overline{\m}= (\fkq_0:_R\m)/(a)$,  $$(\fkq_0:_R\m){\cdot}\Omega^i_R(k) = \fkq_0{\cdot} \Omega_R^i(k)$$ for all $i \ge d$. Therefore, by Theorem  \ref{pro1} (2), $(\fkq:_R\m){\cdot}\Omega^i_R(k) = \fkq{\cdot} \Omega_R^i(k)$ for every parameter ideal $\fkq$ of $R$, because $\Omega^i_R(k)$ is a MCM $R$-module.
\end{proof}

Let us pose one question.

\begin{ques}\label{1.5.1a}
Suppose that $R$ is not a regular local ring. Does the equality 
$$(\fkq:_R\m){\cdot}\Omega^i_R(k) = \fkq{\cdot} \Omega_R^i(k)$$ hold true for every $i \ge 0$ and for every parameter ideal $\fkq$ of $R$? As is shown in Theorem \ref{1.5.1}, this is the case, if $i \ge d=\dim R$. Hence, the answer is affirmative, if $d = 2$ (\cite{CP}).
\end{ques}

Let $M$ be a MCM $R$-module. Then we say that $M$ is an {\it Ulrich $R$-module with respect to $\m$}, if $\mu_R(M) = \rme^0_\m(M)$ (see \cite{BHU}, where the different terminology MGMCM (maximally generated MCM module) is used). Ulrich modules play an important role in the representation theory of local and graded algebras. See \cite{GOTWY1, GOTWY2} for a generalization of Ulrich modules, which later we shall  be back to. Here, let us note that a MCM $R$-module $M$ is an Ulrich $R$-module with respect to $\m$ if and only if $\m M = \fkq M$ for some (hence, every) minimal reduction $\fkq$ of $\m$, provided the residue class field $R/\m$ of $R$ is infinite (see, e.g., \cite[Proposition 2.2]{GTT}).  We refer to \cite[Theorem A]{KT} for the ample existence of Ulrich modules with respect to $\m$ over certain two-dimensional normal local rings $(R,\m)$.

\begin{theorem}\label{1.2}
Suppose that $R$ is not a regular local ring and let $M$ be a MCM $R$-module. We set $A = R \ltimes M$. If $M$ is an Ulrich $R$-module with respect to $\m$, then $\rmr_R(M)=\mu_R(M)$ and $\rmr(A) = \rmr(R) + \rmr_R(M)$, so that $(\fkq :_R \m)M = \fkq M$ for every parameter ideal $\fkq$ of $R$. When $R$ has maximal embedding dimension in the sense of \cite{S}, the converse is also true.
\end{theorem}

\begin{proof}  Enlarging the residue class field of $R$ if necessary, we may assume that $R/\m$ is infinite. Let us choose elements $f_1, f_2, \ldots, f_d$ of $\m$ so that $\fkq = (f_1, f_2, \ldots, f_d)$ is a reduction of $\m$. Then, $\fkq$ is a parameter ideal of $R$, and $\m M = \fkq M$,  since $M$ is an Ulrich $R$-module with respect to $\m$ (\cite[Proposition 2.2]{GTT}).  We then have $\rmr_R(M) = \mu_R(M)$, and   $\fkq :_R\m\subseteq \m$,  because $R$ is not a regular local ring. Hence, $(\fkq :_R\m)M = \fkq M$, because $$\fkq M  \subseteq (\fkq:_R\m)M \subseteq \m M = \fkq M.$$ Thus, $\rmr(A) = \rmr(R) + \rmr_R(M)$ by Theorem \ref{pro1}.

Assume that $R$ has maximal embedding dimension and we will show that the converse also holds true. We have $\m^2 = \fkq \m$ for some parameter ideal $\fkq$ of $R$, so that $\m = \fkq :_R\m$, because $R$ is not a regular local ring. If $\rmr(A) = \rmr(R) + \rmr_R(M)$, we then have $$\m M = (\fkq :_R\m)M = \fkq M$$ by Theorem \ref{pro1} (2), whence $M$ is an Ulrich $R$-module with respect to $\m$.
\end{proof}

\begin{remark}\label{1.3}
Unless $R$ has maximal embedding dimension, the second assertion in Theorem \ref{1.2} is not necessarily true. For example, let $(R,\m)$ be a one-dimensional Gorenstein local ring. Assume that $R$ is not a DVR. Then $\rmr(R \ltimes \m) = 3 = \rmr(R) + \rmr_R(\m)$ (see Proposition \ref{typem} and Corollary \ref{2.3a} below), while $\m$ is an Ulrich $R$-module with respect to $\m$ itself if and only if $\m^2 = a \m$ for some $a \in \m$. The last condition is equivalent to saying that $\rme(R)=2$. 
\end{remark}

We note one more example, for which the both cases $\rmr(R\ltimes M) = \rmr(R) + \rmr_R(M)$ and $\rmr(R\ltimes M) = \rmr_R(M)$ are possible, choosing different MCM modules $M$.

\begin{ex}
Let $R = k[[X,Y,Z]]/(Z^2-XY)$, where $k[[X,Y,Z]]$ denotes the formal power series ring over a field $k$. Then, the indecomposable MCM $R$-modules are $\fkp = (x,z)$ and $R$, up-to isomorphisms (here, by $x,y, z$ we denote the images of $X,Y, Z$ in $R$, respectively). Since $\fkp$ is an Ulrich $R$-module with respect to $\m$, by Theorem \ref{1.2} we have $\rmr(R \ltimes \fkp) = 1 + \rmr_R(\fkp) = 3$. Let $M$ be an arbitrary  MCM $R$-module. Then, $M \cong \fkp^{\oplus \ell} \oplus R^{\oplus n}$ for some integers $\ell, n \ge 0$, and $M/\fkq M$ is a faithful $R/\fkq$-module for the parameter ideal $\fkq=(x,y)$ if and only if $n >0$. Therefore, $\rmr(R\ltimes M) = \rmr_R(M) = 2\ell + n$ if $n >0$, while $\rmr(R \ltimes M) = 1+\rmr_R(M) =1+2\ell$ if $n=0$ (see Theorem  \ref{pro1}).
\end{ex}

The generalized notion of Ulrich ideals and modules was introduced by \cite{GOTWY1}. We briefly review the definition. Let $I$ be an $\m$-primary ideal of $R$ and $M$ a MCM $R$-module. Suppose that $I$ contains a parameter ideal $\fkq$ as a reduction. We say that $M$ is an {\em Ulrich $R$-module} with respect to $I$, if $\rme_I^0(M) = \ell_R(M/IM)$ and $M/IM$ is a free $R/I$-module. Notice that the first condition is equivalent to saying that $IM = \fkq M$ and that the second condition is automatically satisfied, when $I = \m$. We say that $I$ is an {\em Ulrich ideal} of $R$, if $I \supsetneq \fkq$, $I^2 = \fkq I$, and $I/I^2$ is a free $R/I$-module. Notice that when $\dim R=1$, every Ulrich ideal of $R$ is an Ulrich $R$-module with respect itself. Ulrich modules and ideals are closely explored by  \cite{GIK, GOTWY1, GOTWY2, GTT2}, and it is known that they enjoy very specific properties. For instance, the syzygy modules $\Omega_R^i(R/I)$~($i \ge d$) for an Ulrich ideal $I$ are Ulrich $R$-modules with respect to $I$.

\begin{thm}
Let $I$ be an Ulrich ideal of $R$ and $M$ an Ulrich $R$-module with respect to $I$. We set $\ell = \mu_R(M)$ and $m= \mu_R(I)$. Then
$$\rmr(R \ltimes M)=\rmr(R) + \rmr_R(M) = \rmr(R/I){\cdot}(\ell + m -d).$$ 
\end{thm}

\begin{proof}
Let $\fkq$ be a parameter ideal of $R$ such that $I^2 =\fkq I$. Then $IM = \fkq M$ because $\rme_I^0(M) = \ell_R(M/IM)$, while $M/IM \cong (R/I)^{\oplus \ell}$ as an $R/I$-module. Therefore, since  $\Ann_{R/\fkq}M/\fkq M = I/\fkq$ and $I/\fkq \cong (R/I)^{\oplus (m - d)}$ as an $R/I$-module (\cite[Lemma 2.3]{GOTWY1}), we have  by Proposition \ref{lemma1}
$$\rmr(R \ltimes M) = \rmr_R(I/\fkq) + \ell{\cdot}\rmr(R/I)= \rmr(R/I) {\cdot}(m-d) + \ell{\cdot}\rmr(R/I) = \rmr(R) + \rmr_R(M),$$
where the last equality follows from  the fact that $\rmr(R) = (m-d){\cdot}\rmr(R/I)$ (see \cite[Theorem 2.5]{GTT2}). 
\end{proof}

\begin{cor}\label{ulrich}
Suppose that $d=1$ and let $I$ be an Ulrich ideal of $R$ with $m = \mu_R(I)$. Then $\rmr(R \ltimes I) = (2m -1){\cdot}\rmr(R/I)$.
\end{cor}

We note a few examples. 

\begin{ex} Let  $k[[t]]$ be the formal power series ring over a field $k$. 
\begin{enumerate}[{\rm (1)}]
\item Let $R = k[[t^3,t^7]]$. Then $\calX_R= \{(t^6-at^7, t^{10}) \mid 0 \ne a \in k\}$ is exactly the set of Ulrich ideals of $R$. For all  $I \in \calX_R$, $R/I$ is a Gorenstein ring, so that $\rmr(R\ltimes I)= 3$   by Proposition \ref{ulrich}.

\item Let $R=k[[t^6, t^{13}, t^{28}]]$. Then the following families
consist of Ulrich ideals of $R$ (\cite[Example 5.7 (3)]{GIK}):
\begin{enumerate}[{\rm (i)}]
\item $\{(t^6+ at^{13}) + \fkc \mid a \in k\}$,
\item $\{(t^{12} + at^{13} + bt^{19}) + \fkc \mid a, b \in k\}$, and
\item $\{(t^{18}+ at^{25}) + \fkc \mid a \in k\}$,
\end{enumerate}
where $\fkc = (t^{24},t^{26},t^{28})$. We have $\mu_R(I) = 3$ and $R/I$ is a Gorenstein ring for all ideals $I$ in these families, whence $\rmr(R\ltimes I) = 5$.
\end{enumerate}
\end{ex}

Suppose that $\dim R =1$. If $R$ possesses maximal embedding dimension $v$ but not a DVR, then for every Ulrich ideal $I$ of $R$, $R/I$ is a Gorenstein ring, and $I$ is minimally generated by $v$ elements (\cite[Corollary 3.2]{GIK}). Therefore, by Corollary \ref{ulrich}, we get the following.

\begin{cor}
Suppose that $\dim R = 1$ and that $R$ is not a $\operatorname{DVR}$. If $R$ has maximal embedding dimension $v$, then $\rmr(R\ltimes I) = 2v -1$ for every Ulrich ideal $I$ of $R$.
\end{cor}

\section{Bounding the supremum $\sup \rmr(R\ltimes M)$}

Let $r >0$ be an integer and set $$\calF_r(R) = \{M \mid M~\text{is~ an}~ R\text{-submodule~ of}~ R^{\oplus r}~\text{and~ a~ maximal~ Cohen-Macaulay}~ R\text{-module}\}.$$
We are now interested in the supremum $\underset{M \in \calF_r(R)}{\sup} \rmr(R \ltimes M)$ and get the following.

\begin{thm}
Let $(R,\m)$ be a Cohen-Macaulay local ring of multiplicity $e$ and let $M \in \calF_r(R)$. Then $\rmr(R\ltimes M) \le \rmr(R)  + re.$
When $\m$ contains a parameter ideal $\fkq$ of $R$ as a reduction and $R$ is not a regular local ring, the equality holds if and only if $M$ is an Ulrich $R$-module with respect to $\m$, possessing rank $r$.
\end{thm}

\begin{proof}
Enlarging the residue class filed $R/\m$ of $R$ if necessary, without loss of generality we may assume that $\m$ contains a parameter ideal $\fkq$ of $R$ as a reduction. We then have
$$r e \ge \rme_{\fkq}^0(M) = \ell_R(M/\fkq M) \ge \ell_R((0):_{M/\fkq M}\m) = \rmr_R(M).$$
Hence $$\rmr(R \ltimes M) \le \rmr(R) + \rmr_R(M) \le  \rmr(R) + r e.$$ Consequently, if $\rmr(R \ltimes M) = \rmr(R)+ re$, then $re = \rmr_R(M)$, that is $re = \rme^0_\fkq(M)$ and $\ell_R(M/\fkq M) = \ell_R((0):_{M/\fkq M}\m)$, which is equivalent to saying that $\dim_RR^{\oplus r}/M < d$ and $\m M = \fkq M$, that is $M$ has rank $r$ and an Ulrich $R$-module with respect to $\m$. Therefore, when $R$ is not a regular local ring, $\rmr(R \ltimes M) = \rmr(R) + \rmr_R(M)$ if and only if $M$ is an Ulrich $R$-module with rank $r$ (see Theorem \ref{1.2}).
\end{proof}

\begin{cor}\label{5.2}
Suppose that $(R,\m)$ is a Cohen-Macaulay local ring of dimension one and multiplicity $e$. Let $\calF$ be the set of $\m$-primary ideals of $R$. 
Then 
$$\underset{I \in \calF}{\sup}~\rmr(R \ltimes I)=
\begin{cases}
1 & \text{if $R$ is a $\operatorname{DVR}$,}\\
\rmr(R) + e & \text{otherwise.}
\end{cases}
$$
\end{cor}

\begin{proof}
We have only to show the existence of an $\fkm$-primary ideal $I$ such that $I$ is an Ulrich $R$-module with respect to $\m$ and $\mu_R(I) = e$. This is known by \cite[Lemma (2.1)]{BHU}. For the sake of completeness, we note a different proof. Let $$A = \bigcup_{n>0} (\fkm^n:\fkm^n)$$ in $\rmQ(R)$. Then $A$ is a birational finite extension of $R$ (see \cite{Lipman}). Since $A \cong I$ for some $\m$-primary ideal $I$ of $R$, it suffices to show that $A$ is an Ulrich $R$-module with respect to $\m$ and $\mu_R(A) = e$. To do this, enlarging the residue class field $R/\m$ of $R$ if necessary, we may assume that $\m$ contains an element $a$ such that $Q=(a)$ is a reduction of $\m$. Then $\m A = aA$ because $A = R[\frac{\m}{a}]$ (\cite{Lipman}), whence $A$ is an Ulrich $R$-module with respect to $\m$. We have $$\mu_R(A) = \ell_R(A/aA) = \rme^0_{Q}(A) = \rme_Q^0(R) = e$$
as wanted.
\end{proof}

\section{The case where $d=1$}
In this section, we focus our attention on the one-dimensional case. Let $(R,\m)$ be a Cohen-Macaulay local ring of dimension one, admitting a fractional canonical ideal $K$. Hence, $K$ is an $R$-submodule of $\overline{R}$ such that $K \cong \rmK_R$ as an $R$-module and $R \subseteq K \subseteq \overline{R}$, where $\overline{R}$ denotes the integral closure of $R$ in the total ring $\rmQ(R)$ of fractions of $R$. The hypothesis about the existence of fractional canonical ideals $K$ is equivalent to saying that $R$ contains an $\m$-primary ideal $I$ such that $I \cong \rmK_R$ as an $R$-module and such that $I$ possesses a reduction $Q=(a)$ generated by a single element $a$ of $R$ (\cite[Corollary 2.8]{GMP}). The latter condition is satisfied, once $\rmQ(\widehat{R})$ is a Gorenstein ring and the field $R/\m$ is infinite. We have $\rmr_R(M) = \mu_R\left(\Hom_R(M,K)\right)$ for every MCM $R$-module $M$ (\cite[Satz 6.10]{HK}). See \cite{GMP, HK} for more details.

First of all, let us begin with the following review of a result of Brennan and Vasconcelos \cite{BV}. We include a brief proof.

\begin{prop}[{\cite[Propositions 2.1, 5.2]{BV}}]\label{2.3b} Let $I$ be a fractional ideal of $R$ and set $I_1 = K:I$. Then the following conditions are equivalent.
\begin{enumerate}
\item[$(1)$] $I:I = R$.
\item[$(2)$] $I_1{\cdot}I = K$. 
\item[$(3)$] $J{\cdot}I = K$ for some fractional ideal $J$ of $R$.
\item[$(4)$] $I/fI$ is a faithful $R/fR$-module for every parameter $f$ of $R$.
\item[$(5)$] $I/fI$ is a faithful $R/fR$-module for some parameter $f$ of $R$.
\end{enumerate}
\end{prop} 

\begin{proof}
(1) $\Leftrightarrow$ (2) This follows from the facts that $K:I_1I = (K:I_1):I = I:I$, and that $K:K=R$. See \cite[Definition 2.4]{HK} and \cite[Bemerkung 2.5 a)]{HK}, respectively.

(3) $\Rightarrow$ (2) Since $JI=K$, we have $J \subseteq I_1 = K:I$, so that $K= JI \subseteq I_1I \subseteq K$, whence $I_1I=K$.

(2) $\Rightarrow$ (3) This is clear.

Since $I_1 \cong \Hom_R(I,K)$, the assertion that $I_1I = K$ is equivalent to saying that the homomorphism $t^I_K :\Hom_R(I,K) \otimes_R I \to K$ is surjective. Therefore, the equivalence between Assertions (1), (4), (5) are special cases of Corollary \ref{1.4} (see \cite[Proposition 5.2]{BV} also). 
\end{proof}

We say that a fractional ideal $I$ of $R$ is {\it closed}, if it satisfies the conditions stated in Proposition \ref{2.3b}. Thanks to Proposition \ref{2.3b} (3), we readily get the following.

\begin{corollary}[{\cite[Corollary 3.2]{BV}}]\label{2.4a}
If $R$ is a Gorenstein ring, then every closed ideal of $R$ is principal.
\end{corollary}

Assertion (2) of the following also follows from Corollary \ref{gor}. Let us note a direct proof.

\begin{thm}\label{2.1}
Suppose that $R$ is a Gorenstein ring and let $I$ be an $\m$-primary ideal of $R$. Then the following assertions hold true.
\begin{enumerate}[$(1)$]
\item
$\rmr(R/I) \le \rmr_R(I) \le 1 + \rmr(R/I)$, 
\item $\rmr(R \ltimes I) = 1 + \rmr_R(I)$, if $\mu_R(I) > 1$.
\end{enumerate}
\end{thm}

\begin{proof}  Take the $R$-dual of the canonical exact sequence 
$$0 \to I \to R \to R/I \to 0$$
of $R$-modules and we get the exact sequence
$$0 \to R \to \Hom_R(I,R) \to \Ext_R^1(R/I,R) \to 0.$$ Hence, $\rmr(R/I) \le \rmr_R(I) \le 1 + \rmr(R/I)$, because $$\rmr_R(I) = \mu_R(\Hom_R(I,R))\ \ \text{and}\ \  \rmr(R/I) = \mu_R(\Ext_R^1(R/I,R))$$
(\cite[Satz 6.10]{HK}). To see the second assertion, suppose that $\mu_R(I) > 1$. Let $\fkq =(a)$ be a parameter ideal of $R$ and set $J = \fkq :_R\m$. Let us write $J = (a,b)$. We then have $J = \fkq  : \m$, and $\m J = \m \fkq $ by \cite{CP}, because $R$ is not a DVR. On the other hand,  by Corollary \ref{2.4a} we have  $R \subsetneq I : I$, since $R$ is a Gorenstein ring and $I$ is not principal. Consequently
$$R \subseteq R : \m \subseteq I : I,$$
since $\ell_R([R:\m]/R)=1$. Therefore, $\frac{b}{a} \in I:I$, because $$R:\m = \frac{1}{a}{\cdot}[\fkq : \m] = \frac{1}{a}{\cdot}(a,b) =R + R\frac{b}{a}.$$ Thus $bI \subseteq aI$, which shows $(\fkq :_R\m) I =(a,b)I \subseteq \fkq I$, so that $$\rmr(R \ltimes I) = \rmr(R) + \rmr_R(I) = 1 + \rmr_R(I)$$ by Theorem \ref{pro1} (2).
\end{proof}

\begin{remark}\label{2.2}
In Theorem \ref{2.1} (1), the equality $\rmr_R(I) =1+ \rmr(R/I)$ does not necessarily hold true. For instance, consider the ideal  $I = (t^8,t^9)$ in the Gorenstein local ring $R=k[[t^4,t^5,t^6]]$. Then $\rmr(R/I) = 2$. Because $t^{-4} \in R : I$, we have $1 \in \m{\cdot}[R : I]$, which shows, identifying $R:I = \Hom_R(I,R)$ in the proof of Assertion (2) of Theorem \ref{2.1},  that $\mu_R(\Hom_R(I, R)) = \mu_R(\Ext_R^1(R/I,R))$. Hence $\rmr_R(I) = \rmr(R/I) = 2$, while  $\rmr(R \ltimes I) = 3$ by Theorem \ref{2.1} (2).
\end{remark}

We however have $\rmr_R(I) =1+ \rmr(R/I)$ for trace ideals $I$, as we show in the following. Let $I$ be an ideal of $R$. Then $I$ is said to be a {\it trace ideal} of $R$, if $$I = \Im \left(\Hom_R(M,R)\otimes_RM \overset{t^M_R}{\to} R\right)$$ for some $R$-module $M$. When $I$ contains a non-zerodivisor of $R$, $I$ is a trace ideal of $R$ if and only if $R:I=I:I$ (see \cite[Lemma 2.3]{L}). Therefore, $\m$-primary trace ideals are not principal.

\begin{prop}\label{2.3}
Suppose that  $R$ is a Gorenstein ring. Let $I$ be an $\m$-primary trace ideal of $R$. Then $\rmr_R(I) = 1+ \rmr(R/I)$ and $\rmr(R\ltimes I)= 2 + \rmr(R/I)$.
\end{prop}

\begin{proof} We have $1 \not\in \m {\cdot} [R:I]$, since $R:I = I : I \subseteq \overline{R}$. Therefore, thanks to the proof of Assertion (2) in Theorem \ref{2.1}, $\rmr_R(I) = 1 + \rmr(R/I)$, so that $\rmr(R\ltimes I) = 2 + \rmr(R/I)$ by Theorem \ref{2.1} (2).
\end{proof}

\begin{Example}[{\cite[Example 3.12]{GIK}}]\label{2.11}
Let $R=k[[t^4, t^5, t^6]]$. Then $R$ is a Gorenstein ring and $$R, \ (t^{8}, t^{9}, t^{10}, t^{11}),\  (t^{6}, t^{8}, t^{9}),\  (t^{5}, t^{6}, t^{8}),\  (t^{4}, t^{5}, t^{6}), \ \left\{ I_a=(t^{4}-at^5, t^{6})\right\}_{a \in k}$$ are all the non-zero trace ideals of $R$. We have $I_a=I_b$, only if $a= b$.
\end{Example}

\begin{prop}\label{typem}
Suppose that $R$ is a not a $\mathrm{DVR}$. Then $\m$ is a trace ideal of $R$ with $\rmr_R(\m) = \rmr(R) +1$ and $\rmr(R \ltimes \m) = 2{\cdot}\rmr(R) + 1$. 
\end{prop}

\begin{proof}  We have $\m :\m = R : \m$, because $R$ is not a DVR, whence $\m$ is a trace ideal of $R$. We take the $K$-dual of the sequence $0 \to \m \to R \to R/\m \to0$ and consider the resulting exact sequence
$$0 \to K \to K : \m \to \Ext_R^1(R/\m,K) \to 0.$$ Then, since $\Ext_R^1(R/\m,K)\cong R/\m$, we get $$\rmr_R(\m) = \mu_R(K:\m) \le \mu_R(K)+1 = \rmr(R)+1.$$ We actually have the equality in the estimation $$\mu_R(K:\m) \le \mu_R(K)+1.$$ To see this, it is enough to show that $\m (K:\m) = \m K$. We have $$K : \m(K:\m) = [K:(K:\m)]:\m =\m :\m$$ and $$K:\m K = (K:K):\m = R :\m.$$ Therefore, since $\m :\m = R : \m$, we get $K : \m(K:\m) = K : \m K$, so that $\m (K:\m) = \m K$. Hence $\rmr_R(\m) = \mu_R(K:\m) = \mu_R(K)+1 = \rmr(R)+1$ as wanted. We have $\rmr(R \ltimes \m) = \rmr(R) + \rmr_R(\m)$ by Theorem \ref{pro1} (2), because $(\fkq:_R\m){\cdot} \m = \fkq{\cdot} \m$ for every parameter ideal $\fkq$ of $R$ (\cite{CP}; see Theorem \ref{1.5.1} also), whence the second assertion follows. 
\end{proof}

\begin{corollary}\label{2.3a}
Let $R$ be a Gorenstein ring which is not a $\mathrm{DVR}$. Then $R \ltimes \m$ is an almost Gorenstein ring in the sense of \cite{GMP}, possessing $\rmr(R\ltimes \m) = 3$.
\end{corollary}

\begin{proof}
See \cite[Theorem 6.5]{GMP} for the assertion that $R \ltimes \m$ is an almost Gorenstein ring.
\end{proof}

Let us give one more result on closed ideals.

\begin{proposition}\label{2.5}
Let $I \subsetneq R$ be a closed ideal of $R$ and set $I_1 =K : I$.  Then $\rmr(R/I) = \mu_R(I_1) = \rmr_R(I)$.
\end{proposition}

\begin{proof} We consider the exact sequence $0 \to K \to I_1 \to \Ext_R^1(R/I,K) \to 0$. It suffices to show $K \subseteq \m I_1$. We have $K :\m I_1 = (K:I_1) : \m$, while $(K:I_1):\m = I:\m \subseteq I:I = R = K : K$. Hence $\m I_1 \supseteq K$ and the assertion follows. \end{proof}

Combining Corollary \ref{1.4}, Proposition \ref{2.3b}, and Proposition \ref{2.5}, we have the following, which is the goal of this paper.

\begin{cor}\label{2.6}
Let $I$ be a fractional ideal of $R$. Then the following conditions are equivalent.
\begin{enumerate}
\item[$(1)$] $\rmr(R \ltimes I)= \rmr_R(I)$.
\item[$(2)$] $I$ is a closed ideal of $R$.
\end{enumerate}
When this is the case, $\rmr(R \ltimes I) = \rmr(R/I)$, if $I \subsetneq R$.
\end{cor}

We close this paper with the following example.

\begin{ex}\label{2.7} Let $k$ be a field. Let $R = k[[t^3,t^4,t^5]]$ and set $I = (t^3, t^4)$. Then $I \cong \rmK_R$, and $I$ is a closed ideal of $R$ with $\rmr(R)=2$ and $\rmr(R\ltimes I) = \rmr_R(I) = 1$. We have $\rmr(R \ltimes J ) = 1 + \rmr_R(J) =3$ for $J = (t^3,t^5)$. The maximal ideal $\m$ of $R$ is an Ulrich $R$-module, and $\rmr(R \ltimes \m) = 2 + \rmr_R(\m) = 5$ by Theorem \ref{1.2}, since $\rmr_R(\m) = \rmr(R) + 1=3$ by Proposition \ref{typem}. See Corollary  \ref{2.3.3} for more details.
\end{ex}

%%%%%%%%%%%%%%%%%%%%%%%%%%%%%%%%%%%%%%%%%%%

%\begin{abstract}

%\end{abstract}

%%%%%%%%%%%%%%%%%%%%%%%%%%%%%%%%%%%%%%%%%%%%%%%%%%%%%%%%%%%%%%%%%%%%%%%%%%%%%%%%%%%%%%

%\addcontentsline{toc}{section}{references}

\end{document}